\documentclass[12pt,reqno]{amsart}
\usepackage{amsthm}
\usepackage{amscd}
\usepackage{amsfonts}
\usepackage{amssymb}
\usepackage{amsgen}
\usepackage{amsmath}
\usepackage{amsopn}
\usepackage{url}

\theoremstyle{plain}
\newtheorem{thm}{Theorem}[section]
\newtheorem{lem}[thm]{Lemma}
\newtheorem{prop}[thm]{Proposition}
\newtheorem{cor}[thm]{Corollary}
\newtheorem{conj}[thm]{Conjecture}
\theoremstyle{definition}

\newtheorem{eg}[thm]{Example}

 \newcommand{\nc}{\newcommand}

 \nc{\frakP}{{\mathfrak P}}
 \nc{\Z}{{\mathbb Z}}
 \nc{\R}{{\mathbb R}}
 \nc{\N}{{\mathbb N}}
 \nc{\ZN}{{{\mathbb N}_0}}
 \nc{\Q}{{\mathbb Q}}
 \nc{\CC}{{\mathbb C}}

 \nc{\calA}{{\mathcal A}}
 \nc{\calB}{{\mathcal B}}
 \nc{\calH}{{\mathcal H}}
 \nc{\calP}{{\mathcal P}}
 \nc{\calO}{{\mathcal O}}
 \nc{\gam}{{\gamma}}
 \nc{\gG}{{\Gamma}}
 \nc{\om}{{\omega}}
 \nc{\vep}{{\varepsilon}}
 \nc{\ga}{{\alpha}}
 \nc{\gl}{{\lambda}}
 \nc{\gb}{{\beta}}
 \nc{\gd}{{\delta}}
 \nc{\bfk}{{\bf k}}
 \nc{\bfl}{{\bf l}}
 \nc{\bfb}{{\bf b}}
 \nc{\bfs}{{\bf s}}
 \nc{\bft}{{\bf t}}
 \nc{\gs}{{\sigma}}
 \nc{\gth}{{\theta}}
 \nc{\gS}{{\Sigma}}
 \nc{\gk}{{\kappa}}
  \nc{\gz}{{\zeta}}
 \nc{\tgz}{{\tilde{\zeta}}}
 \nc{\gO}{{\Omega}}
 \nc{\sif}{{\mathcal S}}
 \nc{\gt}{{\tau}}
 \nc{\Lra}{\Longrightarrow}
 \nc{\lra}{\longrightarrow}
 \nc{\lmaps}{\longmapsto}
 \nc{\fS}{{\mathfrak S}}
 \nc{\DD}{{\mathfrak D}}
 \nc{\Llra}{\Longleftrightarrow}
 \nc{\ol}{\overline}
 \nc{\ola}{\overleftarrow}
 \nc{\lms}{\longmapsto}
 \nc{\cv}{{{\mathsf c}{\mathsf v}}}
 \nc{\zq}{{\zeta_q}}
 \nc\qup{{q\uparrow 1}}
 \nc{\us}{\underset}
 \nc{\tn}{{\tilde{n}}}
 \nc{\gD}{{\Delta}}
 \nc{\bi}{{\bf i}}
 \nc{\bfone}{{\bf 1}}

\DeclareMathOperator{\hht}{{ht}}

\DeclareMathOperator{\Orb}{{Orb}}
\DeclareMathOperator{\DW}{{DW}}
\DeclareMathOperator{\Stab}{{Stab}}

\begin{document}
\title[Super Congruences Involving AMHS]{Super Congruences Involving Multiple Harmonic Sums and Bernoulli Numbers}

\author{Kevin Chen and Jianqiang Zhao}
\address{Department of Mathematics, The Bishop's School, La Jolla, CA 92037}
\date{}

\email{chenk@bishops.com}\email{zhaoj@ihes.fr}

\subjclass[2010]{11A07, 11B68}

\keywords{Multiple harmonic sums, finite multiple zeta values, Bernoulli numbers, super congruences}

\maketitle
\allowdisplaybreaks

\begin{abstract}
Let $m$, $r$ and $n$ be positive integers. We denote by $\bfk\vdash n$
any tuple of odd positive integers $\bfk=(k_1,\dots,k_t)$ such that
$k_1+\dots+k_t=n$ and $k_j\ge 3$ for all $j$. In this paper we prove that
for every sufficiently large prime $p$
$$
\sum_{\substack{l_1+l_2+\cdots+l_n=mp^r\\ p\nmid l_1 l_2 \cdots l_n }}
\frac1{l_1l_2\cdots l_n} \equiv p^{r-1} \sum_{\bfk\vdash n} C_{m,\bfk} B_{p-\bfk} \pmod{p^r}
$$
where $B_{p-\bfk}=B_{p-k_1}B_{p-k_2}\cdots B_{p-k_t}$ are products of Bernoulli numbers
and the coefficients $C_{m,\bfk}$ are polynomials of $m$ independent of $p$ and $r$.
This generalizes previous results by many different authors and
confirms a conjecture by the authors and their collaborators.
\end{abstract}

\section{Introduction}
The Bernoulli numbers, defined by the generating series
$$\frac{t}{e^t-1}=\sum_{k=0}^\infty B_k\frac{t^k}{k!},$$
have a long and intriguing history in the study of number theory,
with over 3000 related papers written so far according to the online
Bernoulli Number archive  maintained by Dilcher and Slavutskii \cite{DilcherSl}.
In modern mathematics, the Bernoulli numbers
have appeared in the Euler-Maclaurin summation formula, Herbrand's Theorem
concerning the class group of cyclotomic number fields, and even
the Kervaire--Milnor formula in topology.

Well-documented history indicates that Jakob Bernoulli, after whom the Bernoulli numbers are named, was very proud of his discovery that sums of powers of positive integers can be quickly calculated by using these numbers. This result was independently discovered by Seki around the same time \cite{AIK}. By using Fermat's Little Theorem, the formula further leads to many congruences and even super congruences involving multiple harmonic sums, which were first studied independently by the second author in \cite{Zhao2008a,Zhao2011c} and Hoffman in \cite{Hoffman2015}.
See \cite[Ch.~8]{Zhao2015a} for more details.

Let $\N$ and $\N_0$ be the set of positive integers and nonnegative integers, respectively.
For any $n,d\in\N$ and $\bfs=(s_1,\dots,s_d)\in\N^d$ we define
the \emph{multiple harmonic sums} (MHSs) and their $p$-restricted version for primes $p$ by
\begin{align*}
\calH_n(\bfs):= \sum_{0<k_1<\cdot<k_d<n} \frac{1}{k_1^{s_1}\dots k_d^{s_d}},\quad
\calH_n^{(p)}(\bfs):= \sum_{\substack{0<k_1<\cdot<k_d<n\\ p\nmid k_1,\dots, p\nmid k_d}} \frac{1}{k_1^{s_1}\dots k_d^{s_d}}.
\end{align*}
Here, $d$ is called the depth and $|\bfs|:=s_1+\dots+s_d$ the weight of the MHS.
For example, $\calH_{n+1}(1)$ is often called the $n$th harmonic number.
In general, as $n\to \infty$ we see that $\calH_n(\bfs)\to \zeta(\bfs)$ which are
the multiple zeta values (MZVs).

More than a decade ago, the second author discovered the curious congruence (see \cite{Zhao2007b})
\begin{equation}\label{equ:BaseCongruence}
  \sum_{\substack{i+j+k=p\\ i,j,k>0}} \frac1{ijk} \equiv -2 B_{p-3} \pmod{p}
\end{equation}
for all primes $p\ge 3$. Since then several different types of generalizations
have been found, see, for e.g.
\cite{MTWZ,ShenCai2012b,Wang2014b,WangCa2014,XiaCa2010,Zhao2014,ZhouCa2007}.
In this paper, we will concentrate on congruences of the following type of sums.
Let $\calP_p$ be the set of positive integers not divisible by $p$. For
all positive integers $r$ and $m$ such that $p\nmid m$, define
\begin{align*}
R_n^{(m)}(p^r):=&\,\sum_{\substack{l_1+l_2+\dots+l_n=mp^r\\ l_1,\dots,l_n\in \calP_p }}
    \frac{1}{l_1 l_2\dots l_n} ,\\
S_n^{(m)} (p^r):=&\, \sum_{\substack{l_1+l_2+\dots+l_n=mp^r\\ p^r>l_1,\dots,l_n\in \calP_p}}
    \frac{1}{l_1 l_2\dots l_n}.
\end{align*}

To put these sums into proper framework, we now recall briefly the definition of
the finite MZVs.
Let $\frakP$ be the set of rational primes. To study the congruences of MHSs,
Kaneko and Zagier \cite{KanekoZa2013} consider the following
ring structure\footnote{More precisely, they consider only the case when $\ell=1$.}
first used by Kontsevich \cite{Kontsevich2009}:
\begin{equation*}
\calA_\ell:=\prod_{p\in\frakP}(\Z/p^\ell\Z)\bigg/\bigoplus_{p\in\frakP}(\Z/p^\ell\Z).
\end{equation*}
Two elements in $\calA_\ell$ are the same if they differ at only finitely
many components. For simplicity, we often write $p^r$ for the element
$\big(p^r\big)_{p\in\frakP}\in\calA_\ell$ for all positive integers $r<\ell$.
For other properties and facts of $\calA_\ell$
we refer the interested reader to \cite[Ch.~8]{Zhao2015a}.

One now defines the finite MZVs as the following elements in $\calA_\ell$:
\begin{equation*}
\zeta_{\calA_\ell}(\bfs):=\Big(  \calH_p(\bfs) \pmod{p^\ell} \Big)_{p\in\frakP}.
\end{equation*}
It turns out that Bernoulli numbers often play important roles in the study of
finite MZVs, as witnessed by the following result (see \cite[p. 1332]{ZhouCa2007}):
 \begin{alignat*}{3}
\zeta_{\calA_3}(1_n) = &\, (-1)^{n-1}\frac{(n+1)}{2} \gb_{n+2}\cdot p^2
    \quad \ && \text{if $2\nmid n$;} \\
\zeta_{\calA_2}(1_n) = &\,  (-1)^n \gb_{n+1}\cdot p
    \quad && \text{if $2|n$}, \phantom{\frac12}
\end{alignat*}
where $1_n$ is the string $(1,\dots,1)$ with $1$ repeating $n$ times,
and $\gb_k:=\big(-B_{p-k}/k \pmod{p}\big)_{p>k}\in\calA_1$ is
the so-called $\calA$-Bernoulli number, which is the finite analog of $\zeta(k)$.
Note that $\gb_k=0$ for all even positive integers $k$
while it is still a mystery whether $\gb_k\ne 0$ for all odd integers $k>2$.

In \cite{MTWZ}, the second author and his collaborators made the following conjecture.
\begin{conj}\label{conj:RS1}
For any $m,n\in\N$, both $R_n^{(m,1)}$ and $S_n^{(m,1)}$ are elements
in the sub-algebra of $\calA_1$ generated by the $\calA$-Bernoulli numbers.
\end{conj}
In this paper, we will prove this conjecture. More precisely, we have
%\begin{mainthm}\label{thm:main}

\medskip
\noindent{\bf Main Theorem.} {\sl Let $m$, $r$ and $n$ be positive integers. We denote by $\bfk\vdash n$
any tuple of odd positive integers $\bfk=(k_1,\dots,k_t)$ such that
$k_1+\dots+k_t=n$ and $k_j\ge 3$ for all $j$. Then for every sufficiently large prime $p$
\begin{align}
R_n^{(m)}(p^r)\equiv &\sum_{\substack{l_1+l_2+\cdots+l_n=mp^r\\ p\nmid l_1 l_2 \cdots l_n }}
\frac1{l_1l_2\cdots l_n} \equiv p^{r-1} \sum_{\bfk\vdash n} C_{m,\bfk} B_{p-\bfk} \pmod{p^r}, \label{equ:mainR}\\ S_n^{(m)}(p^r)\equiv &\sum_{\substack{l_1+l_2+\cdots+l_n=mp^r\\ p\nmid l_1 l_2 \cdots l_n }}
\frac1{l_1l_2\cdots l_n} \equiv p^{r-1} \sum_{\bfk\vdash n} C'_{m,\bfk} B_{p-\bfk}\pmod{p^r}, \label{equ:mainS}
\end{align}
where $B_{p-\bfk}=B_{p-k_1}B_{p-k_2}\cdots B_{p-k_t}$ are products of Bernoulli numbers and the coefficients $C_{m,\bfk}$ and $C'_{m,\bfk}$ are polynomials of $m$ independent of $p$ and $r$.
}
%\end{mainthm}

\medskip
The coefficients $C_{m,\bfk}$ and $C'_{m,\bfk}$ are intimately related, see Conjecture~\ref{conj:RSr=1}.

As a side remark, in our numerical computation, it is crucial to use some generating functions of $R_{n}^{(m)}$
and $S_{n}^{(m)}$, which are certain products of a finite variation of the $p$-restricted classical polylogarithm function. Unfortunately, it seems difficult to use these generating functions to obtain our main result of this paper.

%\medskip
%\textbf{Acknowledgement.}

\section{Preliminary lemmas}\label{sec:prel}
In this section, we collect some useful results to be applied in the rest of the paper.

\begin{lem}\label{lem:Uab} \emph{(cf.\ \cite[Lemma 3.4]{MTWZ})}
Let $p$ be a prime, $\gk,s_1,\dots,s_d$ be positive integers, and $\ga$ a non-negative integer.
We define the un-ordered sum
\begin{equation*}
U_{\ga;\gk}^{(p)}(s_1,\dots,s_d):=\sum_{\substack{\ga p<l_1,\dots,l_d<(\ga+\gk)p \\
               l_1,\dots,l_d\in \calP_p,\ l_i\ne l_j \forall i\ne j  }}
\frac{1}{l_1^{s_1}\cdots l_d^{s_d}}.
\end{equation*}
If the weight $w=s_1+\dots+s_d\le p-3$ then we have
\begin{equation*}
U_{\ga;\gk}^{(p)}(s_1,\dots,s_d)\equiv
 (-1)^{d-1} (d-1)! \frac{\gk w}{w+1} B_{p-w-1} \cdot p \pmod{p^2}.
\end{equation*}
\end{lem}

\begin{lem} \label{lem:Srecurrence}
Suppose $a,k,m,n,r\in\N$ and $p$ is a prime. Set
\begin{equation*}
\gam^{(m)}_n(a):=(-1)^{m+a}\binom{n-2}{m-1} \frac{(a-1)!(n-1-a)!}{(n-1)!}.
\end{equation*}
If $k<n<p-1$ then we have
\begin{enumerate}
  \item[\upshape (i)] $S_n^{(k)}(p^{r})\equiv (-1)^n S_n^{(n-k)}(p^{r})$ \text{\rm{(mod $p^r$)}};
  \item[\upshape (ii)] $\displaystyle S_n^{(m)}(p^{r+1})\equiv p\sum_{a=1}^{n-1}
   \gam^{(m)}_n(a) S_n^{(a)}(p^{r})  \pmod{p^{r+1}};$
  \item[\upshape (iii)]   $\displaystyle S_n^{(m)}(p^{r+1}) \equiv (-1)^{m-1}\binom{n-2}{m-1} S_n^{(1)}(p^2) p^{r-1}   \pmod{p^{r+1}}.$
\end{enumerate}
\end{lem}
\begin{proof}
(i) and (ii) follow from \cite[Lemma 2.3]{MTWZ} while (iii) from \cite[Lemma 2.2]{CHQWZ}.
\end{proof}

\begin{lem}  \label{lem:RmS1} \emph{(\cite[Proposition\ 2.3]{CHQWZ})}
Let $m,n,r\in\N$. For all $r\ge 2$, we have
\begin{equation*}
R_n^{(m,r)} = m \cdot S_n^{(1,2)} p^{r-2}\in\calA_r.
\end{equation*}
\end{lem}

\begin{lem}\label{lem:Sm=Rms}
Suppose $m,n,r\in\N$. Then we have
\begin{equation}\label{equ:S2R}
S_n^{(m,r)}=\sum_{k=0}^{m-1}(-1)^k \binom{n}{k} R_{n}^{(m-k,r)}\in\calA_r.
\end{equation}
\end{lem}
\begin{proof}
Equation \eqref{equ:S2R} can be proved using the Inclusion-Exclusion
Principle similar to the proof of \cite[Lemma 1]{Wang2015}. Indeed, for all
primes $p$
\begin{align*}
S_{n}^{(m)}(p^r)&\,=\sum_{\substack{l_1+\cdots +l_n=mp^r\\ l_1,\cdots, l_n\in\calP_p}}
 \frac{1}{l_1\cdots l_n}
 +\sum_{k=1}^{m-1}(-1)^{k}
 \sum_{\substack{1\le a_1<\cdots <a_{k}\le n\\
l_1+\cdots +l_n=mp^r\\ l_1,\cdots, l_n\in\calP_p\\
  l_{a_1}>p^r ,\cdots, l_{a_k}>p^r }}
 \frac{1}{l_1\cdots l_n}\\
&\,= \sum_{k=0}^{m-1}(-1)^k\binom{n}{k}
\sum_{\substack{l_1+\cdots +l_n=(m-k)p^r\\ l_1,\cdots, l_n\in\calP_p}}
\frac{1}{(l_1+p^r) \cdots (l_k+p^r)l_{k+1}\cdots l_n}\\
&\equiv \sum_{k=0}^{m-1} (-1)^k\binom{n}{k}
\sum_{\substack{l_1+\cdots +l_n=(m-k)p^r\\ l_1,\cdots, l_n\in\calP_p}}
\frac{1}{l_1 \cdots l_n} \pmod{p^r}\\
&\equiv \sum_{k=0}^{m-1}(-1)^k \binom{n}{k} R_{n}^{(m-k)}(p^r) \pmod{p^r},
\end{align*}
as desired.
\end{proof}

We see immediately from Lemmas \ref{lem:RmS1} and \ref{lem:Sm=Rms} that the proof of
the Main Theorem is reduced to its special case of $S_n^{(1,2)}$. The idea is
to compute $R_{n}^{(m,1)}$ first, which leads to $S_n^{(m,1)}$ by the Lemma~\ref{lem:Sm=Rms}.
Then $S_n^{(1,2)}$ can be determined using $S_n^{(m,1)}$ by Lemma~\ref{lem:Srecurrence} (ii).

For the convenience of numerical computation, we list some of the relevant known results.

\begin{lem}\label{lem:R1S1} \emph{(\cite[Main Theorem]{ZhouCa2007})}
Let $n>1$ be positive integer. Then we have
\begin{equation*}
S_n^{(1,1)} =R_n^{(1,1)}=
\left\{
  \begin{array}{ll}
    \displaystyle n! \gb_n   &  \quad \hbox{if $2\nmid n$;} \\
     \displaystyle 0  &  \quad  \hbox{if $2\mid n$.}
  \end{array}
\right.
\end{equation*}
\end{lem}

\begin{lem}\label{lem:R2S2}
Let $n>1$ be positive integer. Then we have
\begin{equation*}
R_n^{(2,1)}=
\left\{
  \begin{array}{ll}
    \displaystyle  \frac{(n+1)!}{2} \gb_n   &  \quad \hbox{if $2\nmid n$;} \\
     \displaystyle \frac{n!}{2}\sum_{a+b\vdash n}\gb_a\gb_b  & \quad  \hbox{if $2\mid n$,}
  \end{array}
\right.
\end{equation*}
and
\begin{equation*}
S_n^{(2,1)}=
\left\{
  \begin{array}{ll}
    \displaystyle -\frac{n-1}{2}n! \gb_n   &  \quad \hbox{if $2\nmid n$;} \\
     \displaystyle \frac{n!}{2}\sum_{a+b\vdash n}\gb_a\gb_b  & \quad  \hbox{if $2\mid n$.}
  \end{array}
\right.
\end{equation*}
\end{lem}

\begin{proof}
The odd cases follow from \cite[Lemma 3.5 and Cor 3.6]{MTWZ} respectively.
The even cases are proved in \cite[Theorem\ 1 and Corollary\ 1]{Wang2015}.
\end{proof}

\begin{lem}\label{lem:R3S3}
Let $n>1$ be positive integer. Then we have
\begin{equation*}
R_n^{(3,1)}=
\left\{
  \begin{array}{ll}
    \displaystyle  {n+2 \choose 3}\cdot (n-1)! \gb_n
   +\frac{n!}{6}\sum_{\substack{a+b+c\vdash n}}
   \gb_a\gb_b \gb_c   & \quad \hbox{if $2\nmid n$;} \\
   \displaystyle \frac{n!(n+2)}{4}\sum_{a+b\vdash n} \gb_a\gb_b &\quad \hbox{if $2\mid n$,}
  \end{array}
\right.
\end{equation*}
and
\begin{equation*}
S_n^{(3,1)}=
\left\{
  \begin{array}{ll}
    \displaystyle  {n \choose 3}\cdot (n-1)!\gb_n
    +\frac{n!}{6}\sum_{\substack{a+b+c\vdash n}}
     \gb_a\gb_b\gb_c   & \quad \hbox{if $2\nmid n$;} \\
     \displaystyle -\frac{n!(n-2)}{4}\sum_{a+b\vdash n} \gb_a\gb_b &\quad   \hbox{if $2\mid n$.}
  \end{array}
\right.
\end{equation*}
\end{lem}
\begin{proof}
The odd cases of follow from \cite[Lemma 3.7 and Corollary 3.7]{MTWZ} respectively.
The even cases are essentially proved in \cite[Theorem 2 and Corollary 2]{Wang2015}. We only
need to observe that if $n$ is even then by exchanging the indices $a$ and $b$
in half of the sums, we get
\begin{alignat*}{3}
R_n^{(3)}(p) \equiv &\,
\frac{n!}{6}\sum_{a+b\vdash n}(2n-a+3)\frac{B_{p-a}B_{p-b}}{ab} &&\pmod{p}\\
\equiv&\, \frac{n!}{12}\sum_{a+b\vdash n} (4n-a-b+6)\frac{B_{p-a}B_{p-b}}{ab} &&\pmod{p}\\
\equiv&\, \frac{n!(n+2)}{4}\sum_{a+b\vdash n}  \frac{B_{p-a}B_{p-b}}{ab}  &&\pmod{p}
\end{alignat*}
since $a+b=n$. Similarly
\begin{alignat*}{3}
S_n^{(3)}(p) \equiv &\,
-\frac{n!}{6}\sum_{a+b\vdash n} (n+a-3)\frac{B_{p-a}B_{p-b}}{ab}  &&\pmod{p}\\
\equiv&\,-\frac{n!}{12}\sum_{a+b\vdash n} (2n+a+b-6)\frac{B_{p-a}B_{p-b}}{ab} &&\pmod{p}\\
\equiv&\,-\frac{n!(n-2)}{4}\sum_{a+b\vdash n}  \frac{B_{p-a}B_{p-b}}{ab} &&\pmod{p}
\end{alignat*}
as desired.
\end{proof}

\section{Sums related to multiple harmonic sums}
We are now ready to consider the sums $R^{(m,r)}_n$.
The key step is to compute $R^{(m,1)}_n$ for $m\le n/2$, which we now transform using MHSs.
By the definition, for all primes $p$, we have
\begin{align*}
R_{n}^{(m)}(p) =&\, \frac1{mp}\sum_{\substack{l_1+l_2+\dots+l_n=mp\\ l_1,\dots,l_n\in \calP_p }}
    \frac{l_1+l_2+\dots+l_n}{l_1 l_2\dots l_n }  \\
 =&\, \frac{n}{mp}\sum_{\substack{u_{n-1}=l_1+l_2+\dots+l_{n-1}<mp\\ l_1,\dots,l_{n-1}, u_{n-1}\in \calP_p}}
    \frac{1}{l_1 l_2\dots l_{n-1}} \quad(\text{by symmetry of $l_1,\dots,l_n$})\\
 =&\, \frac{n}{mp}\sum_{\substack{u_{n-1}=l_1+l_2+\dots+l_{n-1}<mp\\ l_1,\dots,l_{n-1},u_{n-1}\in \calP_p }}
    \frac{l_1+l_2+\dots+l_{n-1}}{l_1 l_2\dots l_{n-1} u_{n-1}}\\
 =&\, \frac{n(n-1)}{mp}\sum_{\substack{u_{n-2}=l_1+l_2+\dots+l_{n-2}<u_{n-1}<mp \\ l_1,\dots,l_{n-2}\in \calP_p \\ u_{n-1}-u_{n-2},u_{n-1}\in \calP_p}}
    \frac{1}{l_1 l_2\dots l_{n-2} u_{n-1}}.
\end{align*}
Continuing this process by using the substitution $u_j=l_1+l_2+\dots+l_j$ for
each $j=n-3,\dots,2,1$, we arrive at
\begin{equation*}
R_{n}^{(m)}(p) =\frac{n!}{mp} \sum_{\substack{0< u_1<\dots<u_{n-1}<mp \\  u_1,u_2-u_1,\dots,u_{n-1}-u_{n-2},u_{n-1} \in \calP_p}}
    \frac{1}{u_1 u_2\dots u_{n-1}} .
\end{equation*}
Observe that the indices $u_j$ ($j=2,\dots,n-2$) are allowed to be multiples of $p$. Thus we set
\begin{equation*}
T_{n,\ell}^{(m)}(p)\,:=
\sum_{\substack{2\le a_1<\cdots<a_{\ell-1}\le n-2\\ 1\le k_1<\dots<k_{\ell-1}<m}} \
\sum_{\substack{0< u_1<\dots<u_{n-1}<mp
    \\ u_{a_1}=k_1p, \dots,u_{a_{\ell-1}}=k_{\ell-1} p, \\
    u_j \in \calP_p \ \forall j\ne a_1,\dots,a_{\ell-1} \\
      u_2-u_1,\dots,u_{n-1}-u_{n-2} \in \calP_p}} \frac{1}{u_1\dots u_{n-1}}.
\end{equation*}
In this sum, the indices $u_1,\dots,u_{n-1}$ are divided into $\ell$-parts by $p$-multiples
so that the indices inside each part (excluding the boundaries) are all prime to $p$.
Hence we can rewrite
\begin{equation}\label{equ:RtoT}
R_{n}^{(m)}(p) =\frac{n!}{mp}\sum_{1\le \ell< n/2}  T_{n,\ell}^{(m)}(p).
\end{equation}
So we are naturally led to the study of the following sums.
Let $\ga\in\N_0$, $\gk,n\in\N$ and $p$ be a prime. Suppose $n>1$. Define
\begin{equation*}
 \Xi_{\ga;\gk}^{(p)}(n):=\sum_{\substack{\ga p< u_1<\dots<u_{n-1}< (\ga+\gk) p
    \\ u_1,u_2,\dots,u_{n-1} \in \calP_p \\
      u_2-u_1,\dots,u_{n-1}-u_{n-2} \in \calP_p}} \frac{1}{u_1\dots u_{n-1}}.
\end{equation*}

For convenience, in the above sum if the difference between two adjacent indices is
a multiple of $p$ (which is of course not allowed in the definition of $\Xi$)
we then say there is a $p$-gap between this pair of indices.
Let $\ga\in\N_0$ and $\gk\in\N$. For all $1\le g\le \min\{\gk-1,n-2\}$,
define the sum in which at least $g$ $p$-gaps appear by
\begin{equation}\label{equ:defnP}
P_{\ga;\gk}^{g;p} (n):=\sum_{1< b_1<\cdots<b_g< n}
\sum_{\substack{\ga p< u_1<\dots<u_{n-1}<(\ga+\gk) p
    \\ u_1,u_2,\dots,u_{n-1} \in \calP_p \\
      p|(u_{b_1}-u_{b_1-1}),\dots,  p|(u_{b_g}-u_{b_g-1})}} \frac{1}{u_1\dots u_{n-1}}.
\end{equation}

The following technical result is crucial in the proof of our Main Theorem.
\begin{prop}\label{prop:Pg}
Let $\ga\in\N_0$, $\gk,n\in\N$. Then,
for all $1\le g\le \min\{\gk-1,n-2\}$, we have
\begin{equation}\label{equ:ConjPg}
P_{\ga;\gk}^{g;p} (n) \equiv P_{0;\gk}^{g;p} (n)
    \equiv  -(-1)^g  \binom{\gk}{g+1} \binom{n-1}{g} \frac{B_{p-n}}{n}p  \pmod{p^2}.
\end{equation}
\end{prop}

We postpone the proof of this proposition to the next section due to its length.
A direct consequence is the following corollary.

\begin{cor}\label{cor:XiExplicit}
Let $\ga,\gk,n\in\N$. Then for all primes $p>n+1$, we have
\begin{equation*}
 \Xi_{\ga;\gk}^{(p)}(n) \equiv\left[\binom{\gk}{n}-\binom{\gk+n-1}{n}\right]
 \frac{B_{p-n}}{n}  p \pmod{p^2}.
\end{equation*}
\end{cor}

\begin{proof}
Set $\gd_{n>j}=1$ if $n>j$ and $\gd_{n>j}=0$ if $n\le j$.
By the Inclusion-Exclusion Principle it is clear that
\begin{align*}
 \Xi_{\ga;\gk}^{(p)}(n)=&\, \frac{U_{\ga;\gk}^{(p)}(1_{n-1})}{(n-1)!}+\sum_{g=1}^{\gk-1} (-1)^g \gd_{n>g+1} P_{\ga;\gk}^{g;p} (n)  & \  \\
 \equiv& \, -\sum_{h=1}^{\gk}   \gd_{n>h} \binom{\gk}{h} \binom{n-1}{h-1} \frac{B_{p-n}}{n}p  &\pmod{p^2}\\
 \equiv& \,\left[\binom{\gk}{n} -\sum_{h=1}^{\gk}  \binom{\gk}{h} \binom{n-1}{h-1}\right]
 \frac{B_{p-n}}{n} p &\pmod{p^2}
\end{align*}
by \eqref{equ:ConjPg} and the congruence (see Lemma \ref{lem:Uab})
\begin{equation*}
 \frac{U_{0;\gk}^{(p)}(1_{n-1})}{(n-1)!}\equiv -\frac{\gk B_{p-n}}{n} p\pmod{p^2}.
\end{equation*}
So the proposition follows immediately from the well-known binomial identity
\begin{equation*}
\sum_{h=1}^{\gk} \binom{\gk}{h} \binom{n-1}{h-1}
=\sum_{h=1}^{\gk} \binom{\gk}{h} \binom{n-1}{n-h}=\binom{\gk+n-1}{n}.
\end{equation*}
\end{proof}

By Corollary~\ref{cor:XiExplicit}, it is easy to see that for any fixed $\ell<n/2$
\begin{align}
T_{n,\ell}^{(m)}(p)\,&\equiv \sum_{\substack{1\le a_1<\cdots<a_{\ell-1}<n\\ 1\le k_1<\dots<k_{\ell-1}<m}}
\left(\prod_{j=1}^{\ell-1} \frac{1}{k_jp}\right)
\left(\prod_{j=1}^{\ell} \Xi_{k_{j-1};k_j-k_{j-1}}^{(p)}(a_j-a_{j-1}) \right) \notag \\
&\, \equiv  p\sum_{\substack{ k_1+\dots+k_\ell=m \\ k_1,\dots,k_\ell \ge 1\\
a_1+\dots+a_\ell\vdash n}}
\prod_{j=1}^{\ell-1}\frac{1}{k_1+\cdots+k_j}
\prod_{j=1}^\ell  \left[\binom{k_j}{a_j}-\binom{k_j+a_j-1}{a_j}\right]  \frac{B_{p-a_j}}{a_j}
\label{equ:Tnl}
\end{align}
modulo $p^2$, where we have set $k_0=a_0=0$ and $k_\ell=m, a_\ell=n$.
In the last step above, we have used substitutions
$k_j\to k_1+\cdots+k_j$ and $a_j\to a_1+\cdots+a_j$ for all $j\le \ell-1$.
In view of \eqref{equ:RtoT} and Lemma~\ref{lem:Sm=Rms}, we easily obtain the following
result which confirms Conjecture~\ref{conj:RS1}.
\begin{thm}\label{thm:R1}
For all positive integer $m$ and $n$, we have
\begin{equation*}
R_{n}^{(m,1)} =\frac{n!}{m}
\sum_{\substack{1\le \ell\le \lfloor n/3 \rfloor\\ k_1+\dots+k_\ell=m,k_j\ge 1 \forall j \\
a_1+\dots+a_\ell\vdash n }}
\prod_{j=1}^{\ell-1}\frac{1}{k_1+\cdots+k_j}
\prod_{j=1}^\ell  \left[\binom{k_j+a_j-1}{a_j}-\binom{k_j}{a_j}\right] \gb_{a_j}.
\end{equation*}
\end{thm}

\section{Some numerical examples}
Using the formula of Theorem~\ref{thm:R1}, we obtain the following results
which extend those in Lemmas~\ref{lem:R1S1}, \ref{lem:R2S2} and \ref{lem:R3S3}.
To guarantee accuracy, we have checked these congruences for $m,n\le 20$ and primes $p<100$ using Maple.
\begin{cor}\label{cor:Rm=4}
For any $\gk,m,n\in\N$, we have
\begin{equation}\label{equ:Rn234}
 R^{(m,1)}_3 = 3!m \gb_3,\ R^{(m,1)}_5 = \frac{5!}{3!} m(m^2+5) \gb_5,\
 R^{(m,1)}_7 = \frac{7!}{5!}m(m^4+35m^2+84)\gb_7.
\end{equation}
If $n\ge 9$ is odd then
\begin{align*}
 R^{(4,1)}_n =&\, (n-1)! \binom{n+3}{4} \gb_n
+\frac{n!(n+3)}{2\cdot 3!} \sum_{a+b+c\vdash n} \gb_a\gb_b\gb_c,\\
R^{(5,1)}_n =&\, (n-1)! \binom{n+4}{5}\gb_n
+\frac{n!}{5!}\sum_{a_1+\dots+a_5\vdash n} \gb_{a_1}\cdots\gb_{a_5}\\
 &\, +\frac{n!}{4!}\sum_{a+b+c\vdash n}\left(\frac{n^2}{2}+4n+7+\frac{a^2}{2}-4\binom{3}{a} \right) \gb_a\gb_b\gb_c,\\
R^{(6,1)}_n=&\,
(n-1)! \binom{n+5}{6}\gb_n
+\frac{n!(n+5)}{2\cdot 5!}\sum_{a_1+\dots+a_5\vdash n} \gb_{a_1}\cdots\gb_{a_5}\\
+\frac{n!}{4\cdot 4!}& \sum_{a+b+c\vdash n}\left(
 \frac{n^3}3+a^3+2a^2b+5n^2+5a^2 +\frac{68n}3+30-8\binom{3}{a}(n+5)\right)\gb_a\gb_b\gb_c.
\end{align*}
If $n\ge 2$ is even then
\begin{align*}
 R^{(4,1)}_n =&\,
\frac{n!}{4!}\left(\sum_{a+b\vdash n}\left(\frac32n^2+9n+11+a^2-8\binom{3}{a}\right) \gb_a\gb_b
+\sum_{a+b+c+d\vdash n} \gb_a\gb_b\gb_c\gb_d\right),\\
R^{(5,1)}_n =&\,\frac{n!}{3\cdot 4!}\sum_{a+b\vdash n}  \left(n^3+9n^2+\frac{63}{2}n+30+a^3+6a^2-12(n+4)\binom{3}{a}\right) \gb_a\gb_b \\
&\, + \frac{n!(n+4)}{2\cdot 4!}  \sum_{a+b+c+d\vdash n} \gb_a\gb_b\gb_c\gb_d,\\
R^{(6,1)}_n =&\,\frac{n!}{6!}\sum_{a+b+c+d+e+f\vdash n}
\gb_a\gb_b\gb_c\gb_d\gb_e\gb_f\\
&\,+ \frac{n!}{6} \sum_{a+b\vdash n}  \left[
\frac13 \binom{3}{a}\binom{3}{b}-\frac65\binom{5}{a}
-\frac{n^2+6n-16}{3}\binom{3}{a}+\frac{1}{ 5!}\left( \frac83a^4+\right. \right.\\
&\,  \left. \left.
+25a^3+85a^2+\frac{675n}{2}+274+\frac56a^3n+\frac53n^4+25n^3+
 \frac{255}{2} n^2\right)  \right]  \gb_a\gb_b \\
&\, +  \frac{n!}{144} \sum_{a+b+c+d\vdash n} \left[
 a^2+\frac{3n^2}{4}+\frac{15n}{2}+17
-8\binom{3}{a}\right]  \gb_a\gb_b\gb_c\gb_d.
\end{align*}
\end{cor}

\begin{eg}
When $8\le n\le12$ we get, respectively,
\begin{alignat*}{3}
R^{(4,1)}_8=&\, 16\cdot 8!\gb_3\gb_5,\quad &&
R^{(4,1)}_9= 9!(55 \gb_9+ \gb_3^3),\\
R^{(5,1)}_{10}=&\, 35\cdot 10!(2 \gb_3\gb_7+ \gb_5^2), \quad &&
R^{(5,1)}_{11}= 11!\Big(273\gb_{11}+ \frac{29}2 \gb_3^2\gb_5\Big),\\
R^{(6,1)}_{12}=&\,  12!\Big(333 \gb_3\gb_9+ 321\gb_5\gb_7+ \frac32\gb_3^4\Big).
\end{alignat*}
The first three identities were predicted by \cite[Conjecture 5.1]{MTWZ}. The last two were also discovered numerically
earlier  \cite[Conjecture 7.2]{Zhao2014}.
\end{eg}

\begin{cor}\label{cor:Sm=4}
Let $n\ge 2$ be a positive integer. If $n$ is odd then
\begin{align*}
 S^{(4,1)}_n =&\, - (n-1)! \binom{n}{4} \gb_n
-\frac{n!(n-3)}{12}\sum_{a+b+c\vdash n} \gb_a\gb_b\gb_c,\\
S_n^{(5,1)}=&\, (n-1)! \binom{n}{5}\gb_n
+\frac{n!}{5!}\sum_{a_1+\dots+a_5\vdash n} \gb_{a_1}\cdots\gb_{a_5}\\
 &\, +\frac{n!}{4!}\sum_{a+b+c\vdash n} \left(\frac{n^2}{2}-4n+7+\frac{a^2}{2}-4\binom{3}{a} \right) \gb_a\gb_b\gb_c,\\
S_n^{(6,1)}=&\,
-(n-1)!\binom{n}{6}\gb_n
-\frac{n-5}{2\cdot 5!}\sum_{a_1+\dots+a_5\vdash n} \gb_{a_1}\cdots\gb_{a_5}\\
- \frac1{96}& \sum_{a+b+c\vdash n}\left(
 \frac{n^3}3+a^3-2a^2b-5n^2+5a^2 +\frac{68n}3-30-8\binom{3}{a}(n-5)\right)\gb_a\gb_b\gb_c.
\end{align*}
If $n$ is even then
\begin{align*}
S^{(4,1)}_n=&\,
\frac{n!}{4!}\sum_{a+b\vdash n}\left(\frac32 n^2-9n+11+a^2-8\binom{3}{a}\right)\gb_a\gb_b
 +\frac{n!}{4!}\sum_{a+b+c+d\vdash n} \gb_a\gb_b\gb_c\gb_d,\\
S_n^{(5,1)} =&\,\frac{n!}{144}\sum_{a+b\vdash n}  \left(
-2n^3+18n^2-63n+60-2a^3+12a^2+24(n-4)\binom{3}{a}\right) \gb_a\gb_b \\
&\, -\frac{n!(n-4)}{48} \sum_{a+b+c+d\vdash n} \gb_a\gb_b\gb_c\gb_d,\\
S^{(6,1)}_n =&\,\frac{n!}{6!}\sum_{a+b+c+d+e+f\vdash n}
\gb_a\gb_b\gb_c\gb_d\gb_e\gb_f\\
&\,+ \frac{n!}{6} \sum_{a+b\vdash n}  \left[
\frac13 \binom{3}{a}\binom{3}{b}-\frac65\binom{5}{a}
-\frac{n^2-9n-16}{3}\binom{3}{a}+\frac{1}{ 5!}\left(
 \frac83a^4 \right.\right.\\
&\,  \left.  \left. -25a^3+85a^2-\frac{675n}{2}+274+\frac56a^3n+\frac53n^4-25n^3+
 \frac{255}{2} n^2\right)  \right]  \gb_a\gb_b \\
&\, +  \frac{n!}{144} \sum_{a+b+c+d\vdash n} \left[
 a^2+\frac{3n^2}{4}-\frac{15n}{2}+17
-8\binom{3}{a}\right]  \gb_a\gb_b\gb_c\gb_d.
\end{align*}
\end{cor}
\begin{proof}
By Lemma \ref{lem:Sm=Rms}, we see that
\begin{equation*}
S_n^{(4,1)}=R_{n}^{(4,1)} -n R_{n}^{(3,1)}
    +\binom{n}{2} R_{n}^{(2,1)} -\binom{n}{3} R_{n}^{(1,1)}.
\end{equation*}
Thus the statements concerning $S_n^{(4,1)}$ follow from
Lemma \ref{lem:R1S1}, \ref{lem:R2S2}, \ref{lem:R3S3} and Corollary\ \ref{cor:Rm=4} immediately.
The computation of $S_n^{(5,1)}$ and $S_n^{(6,1)}$ can be done similarly. So we leave them to
the interested reader.
\end{proof}

By comparing the above two corollaries, we can formulate the following conjecture.
\begin{conj}\label{conj:RSr=1}
For all $m,n\in\N$, suppose
$$
R_n^{(m,1)}= n!\sum_{1\le l\le n/3, 2|(n-l)}\
\sum_{a_1+\dots+a_l\vdash n} C(a_1,\dots,a_l) \gb_{a_1}\dots \gb_{a_l}.
$$
Then
$$
S_n^{(m,1)}=n!\sum_{1\le l\le n/3, 2|(n-l)}\
\sum_{a_1+\dots+a_l\vdash n} C(-a_1,\dots,-a_l) \gb_{a_1}\dots \gb_{a_l}.
$$
\end{conj}

\begin{cor}\label{cor:S8-S12}
For all $r\ge 2$, we have
\begin{alignat*}{3}
S_8^{(m,r)}=&\, (-1)^m\binom{6}{m-1}  5376\gb_3\gb_5 p^{r-1}  &&\in\calA_r \quad \forall m\le 7,\\
S_9^{(m,r)}=&\, (-1)^{m-1}\binom{7}{m-1} 36(6088\gb_9+61 \gb_3^3)p^{r-1} &&\in\calA_r\quad \forall m\le 8,\\
S_{10}^{(m,r)}=&\, (-1)^m\binom{8}{m-1}
 223200(\gb_5^2+2\gb_3\gb_7)p^{r-1} &&\in\calA_r\quad \forall m\le 9,\\
S_{11}^{(m,r)} =&\, (-1)^{m-1}\binom{9}{m-1}174240 (122\gb_{11}+3\gb_3^2\gb_5)p^{r-1} &&\in\calA_r\quad \forall m\le 10,\\
S_{12}^{(m,r)} =&\,  (-1)^m\binom{10}{m-1} 47520 ( 896\gb_3\gb_9 + 872\gb_5\gb_7+ 3\gb_3^4)p^{r-1} &&\in\calA_r\quad \forall m\le 11.
\end{alignat*}
\end{cor}
\begin{proof}
Let $p$ be a prime such that $p\ge 17$. By Lemma \ref{lem:R1S1}, \ref{lem:R2S2},
\ref{lem:R3S3}, and Corollary\ \ref{cor:Sm=4},
we have modulo $p$
\begin{align*}
S_8^{(1)}(p) \equiv&\, 0,\
S_8^{(2)}(p) \equiv \frac{8!}{15}B_{p-3}B_{p-5}, \
S_8^{(3)}(p) \equiv -3S_8^{(2)}(p), \
S_8^{(4)}(p) \equiv 4S_8^{(2)}(p),\\
S_9^{(1)}(p) \equiv&\, -8! B_{p-9},\
S_9^{(2)}(p) \equiv  4\cdot 8!B_{p-9}, \\
S_9^{(3)}(p) \equiv&\, -\frac{8!}{18}B_{p-3}^3-\frac{28\cdot 8!}{3}B_{p-9}, \
S_9^{(4)}(p) \equiv \frac{8!}{6}B_{p-3}^3+ 14\cdot 8!B_{p-9},\\
S_{10}^{(1)}(p) \equiv&\, 0,\
S_{10}^{(2)}(p) \equiv  \frac12\cdot 10!
    \left(\frac{B_{p-5}^2}{25}+\frac{2B_{p-3}B_{p-7}}{21}\right) ,\\
S_{10}^{(3)}(p) \equiv &\, -4 S_{10}^{(2)}(p),
S_{10}^{(4)}(p) \equiv 8S_{10}^{(2)}(p)  , \
S_{10}^{(5)}(p) \equiv  -10S_{10}^{(2)}(p) ,\\
S_{11}^{(1)}(p) \equiv &\, -10!B_{p-11},\
S_{11}^{(2)}(p) \equiv 5\cdot 10!B_{p-11},\\
S_{11}^{(3)}(p) \equiv &-\frac{11!}{90}B_{p-3}^2B_{p-5}-15\cdot 10!B_{p-11} ,\\
S_{11}^{(4)}(p) \equiv &\, \frac{2\cdot 11!}{45}B_{p-3}^2B_{p-5}+30\cdot 10!B_{p-11}  , \\
S_{11}^{(5)}(p) \equiv &-\frac{7\cdot 11!}{90}B_{p-3}^2B_{p-5}-42\cdot 10!B_{p-11},\\
S_{12}^{(1)}(p) \equiv &\, 0,\
S_{12}^{(2)}(p) \equiv \frac{12!}{27}B_{p-3}B_{p-9}+\frac{12!}{35}B_{p-5}B_{p-7},\
S_{12}^{(3)}(p) \equiv -5 S_{12}^{(2)}(p),\\
S_{12}^{(4)}(p) \equiv &\,
\frac{40\cdot 12!B_{p-3}B_{p-9}}{81}+ 13\cdot 12!\frac{B_{p-5}B_{p-7}}{35}
+\frac{12!}{24}\frac{B_{p-3}^4}{3^4}, \\
S_{12}^{(5)}(p) \equiv &\,-\frac{70\cdot 12!B_{p-3}B_{p-9}}{81}
-22\cdot 12!\frac{B_{p-5}B_{p-7}}{35}-\frac{12!}{6}\frac{B_{p-3}^4}{3^4},\\
S_{12}^{(6)}(p) \equiv &\, \frac{28\cdot 12!B_{p-3}B_{p-9}}{27}
+26\cdot 12!\frac{B_{p-5}B_{p-7}}{35}+\frac{12!}{4}\frac{B_{p-3}^4}{3^4}.
\end{align*}
Taking $n=8$ and $r=1$ in  Lemma \ref{lem:Srecurrence} (i) and (ii), we get
\begin{align*}
 S_8^{(1)}(p^2)\equiv &\, \frac{2p}7S_8^{(1)}(p)-\frac{p}{21}S_8^{(2)}(p)
 + \frac{2p}{105}S_8^{(3)}(p) -\frac{p}{140}S_8^{(4)}(p)  \\
 \equiv &\,
-\frac{1792}{5}pB_{p-3}B_{p-5} \pmod {p^2}.
\end{align*}
Similarly, taking $9\le n\le 12$ and $r=1$ in Lemma \ref{lem:Srecurrence} (i) and (ii), we see that
\begin{align*}
 S_9^{(1)}(p^2)\equiv&\,  \frac{p}4 S_9^{(1)}(p)-\frac{p}{28}S_9^{(2)}(p)
 + \frac{p}{84}S_9^{(3)}(p) -\frac{p}{140}S_9^{(4)}(p)  \\
 \equiv&\,
-288\left(\frac{761B_{p-9}}9 +\frac{7 B_{p-3}^3}{3^3} \right)p \pmod {p^2},\\
S_{10}^{(1)}(p^2)\equiv&\,  \frac{2p}9 S_{10}^{(1)}(p)-\frac{p}{36}S_{10}^{(2)}(p)
 + \frac{p}{126}S_{10}^{(3)}(p) -\frac{p}{252}S_{10}^{(4)}(p)+\frac{p}{630}S_{10}^{(5)}(p)  \\
\equiv&\,
-194400\left(\frac{B_{p-5}^2}{25}+\frac{2B_{p-3}B_{p-7}}{21}\right)p \pmod {p^2},\\
S_{11}^{(1)}(p^2) \equiv&\,  \frac{p}5 S_{11}^{(1)}(p)-\frac{p}{45}S_{11}^{(2)}(p)
 + \frac{p}{180}S_{11}^{(3)}(p) -\frac{p}{420}S_{11}^{(4)}(p)+\frac{p}{630}S_{11}^{(5)}(p)   \\
\equiv&\,
-174240 \left(\frac{122B_{p-11}}{11}+\frac{3B_{p-3}^2 B_{p-5}}{45}\right)p \pmod {p^2},\\
S_{12}^{(1)}(p^2) \equiv&\,  \frac{2p}{11} S_{12}^{(1)}(p)-\frac{p}{55}S_{12}^{(2)}(p)
 + \frac{2p}{495}S_{12}^{(3)}(p) -\frac{p}{660}S_{12}^{(4)}(p)\\
 & +\frac{p}{1155}S_{12}^{(5)}(p)
 -\frac{p}{2772}S_{12}^{(6)}(p)   \\
\equiv&\,
-47520 \left(\frac{896B_{p-3}B_{p-9}}{27}+\frac{872B_{p-5}B_{p-7}}{35}+\frac{3B_{p-3}^4}{3^4}\right)p \pmod {p^2}.
\end{align*}
Now the corollary follows quickly from Lemma \ref{lem:Srecurrence} (iii).
\end{proof}

\begin{cor} \label{cor:R8-S12}
For all positive integers $m\ge 1$, we have
\begin{align*}
R_8^{(m,1)} &\, = 336 m(m^2+16)(m^2-1) \gb_3\gb_5,  \\
R_9^{(m,1)} &\, = 12\cdot 7!\binom{m+2}{5} \gb_3^3
+72 m(m^6+126 m^4+1869 m^2+3044) \gb_9,  \\
R_{10}^{(m,1)}&\, = 360m(m^2-1)(m^4+71m^2+540)
(2\gb_3 \gb_7+ \gb_5^2),  \\
R_{11}^{(m,1)} &\,= 660\cdot 5! \binom{m+2}{5}(m^2+33) \gb_3^2 \gb_5\\
&\, +110m(m^8+330 m^6+16401 m^4+152900 m^2+193248) \gb_{11} ,\\
R_{12}^{(m,1)} &\, = 55\cdot 9! \binom{m+3}{7} \gb_3^4 \\
&\,+ 11\cdot 6!  \binom{m+1}{3}(m^6+211m^4+6196m^2+32256) \gb_3 \gb_9\\
&\,+11\cdot 6!  \binom{m+1}{3}(m^6+187m^4+6508m^2+31392) \gb_5\gb_7.
\end{align*}
\end{cor}
\begin{proof}
Let $p$ be a prime such that $p\ge 11$.
By Lemma \ref{lem:RmS1} and Corollary\ \ref{cor:S8-S12}, we have
\begin{align*}
R_8^{(m)}(p) \equiv\sum_{a=1}^{7} \binom{m+7-a}{7} S_8^{(a)}(p)
  \equiv \frac{112}{5} m(m^2+16)(m^2-1) B_{p-3}B_{p-5}  \pmod {p}.
\end{align*}
Similarly,
\begin{align*}
R_9^{(m)}(p)&\, \equiv  -\frac{8!}{18}\binom{m+2}{5} B_{p-3}^3
-8m(m^6+126 m^4+1869 m^2+3044) B_{p-9} \pmod {p},\\
R_{10}^{(m)}(p)&\, \equiv  \frac{10!}{10080}m(m^2-1)(m^4+71m^2+540)
\left(\frac{2 B_{p-3} B_{p-7}}{21}+ \frac{B_{p-5}^2}{25} \right) \pmod {p},\\
R_{11}^{(m)}(p)&\,\equiv -88\cdot 5! \binom{m+2}{5}(m^2+33)
B_{p-3}^2 B_{p-5}\\
&\, -10m(m^8+330 m^6+16401 m^4+152900 m^2+193248) B_{p-11} \pmod{p},\\
R_{12}^{(m)}(p)&\,  \equiv \frac{55\cdot 8!}{9} \binom{m+3}{7} B_{p-3}^4\\
&\,+\frac{22\cdot 5!}{9} \binom{m+1}{3}(m^6+211m^4+6196m^2+32256) B_{p-3} B_{p-9}\\
&\,+\frac{66\cdot 4!}{7} \binom{m+1}{3}(m^6+187m^4+6508m^2+31392) B_{p-5} B_{p-7} \pmod{p}.
\end{align*}
The corollary now quickly follows from the definition of $\gb_k$.
\end{proof}

\section{Proof of Proposition \ref{prop:Pg} and Main Theorem}
We first deal with the case $\ga=0$ and rewrite it as a difference of two sums each of which can be computed more easily.

Let $n,\gk, g\in\N$ such that $1\le g\le \min\{\gk-1,n-2\}$. Set $d=n-g-1$.
For any prime $p$, we define
\begin{equation*}
V_{\gk}^{g;p} (n):=\sum_{\substack{0<a_1<\cdots<a_g<\gk\\ 0< b_1\le \cdots\le b_g\le d}}\
\sum_{\substack{0< u_1<\dots<u_d<(\gk-a_g) p
    \\ u_1,u_2,\dots,u_d \in \calP_p}} \frac{1}{u_1\dots u_d u_{b_1}\dots u_{b_g}},
\end{equation*}
and
\begin{align*}
M_{\gk}^{g;p} (n):= &\sum_{\substack{0<a_1<\cdots<a_g<\gk\\ 0< b_1\le \cdots\le b_g\le d}}\,
\sum_{\substack{0< u_1<\dots<u_d<(\gk-a_g) p
    \\ u_1,u_2,\dots,u_d \in \calP_p}} \frac{1}{u_1\dots u_d u_{b_1}\dots u_{b_g}}
    \left(\frac{a_1}{u_{b_1}}+\frac{a_1}{u_{b_1+1}}+ \right.\\
 \cdots  &\left.+\frac{a_1}{u_{b_2}} +\frac{a_2}{u_{b_2}}+\cdots+\frac{a_2}{u_{b_3}}+\frac{a_3}{u_{b_3}}+\cdots+
    \frac{a_{g-1}}{u_{b_g}}+\frac{a_g}{u_{b_g}}+\cdots+ \frac{a_g}{u_d}\right).
\end{align*}

\begin{lem} \label{lem:Vg}
We have
\begin{equation*}
V_{\gk}^{g;p} (n)\equiv (-1)^{g+1}\binom{\gk}{g+1}\binom{n-1}{g}\frac{B_{p-n}}{n} p \pmod{p^2}.
\end{equation*}
\end{lem}
\begin{proof}
Let $d=n-g-1$ and $m\in\N$. For each $0< b_1\le \cdots\le b_g\le d$,
we write $\bfb=(b_1,\dots,b_g)$ and define
\begin{equation*}
K^{(p)}_{d;m}(\bfb):=\sum_{\substack{0< u_1<\dots<u_d<m
    \\ u_1,u_2,\dots,u_d\in \calP_p}} \frac{1}{u_1\dots u_d u_{b_1}\dots u_{b_g}}.
\end{equation*}
Let $[d]^g$ be the set of $g$-tuples of integers in $\{1,\dots,d\}$. Let $\DW(d,n-1)\subset \N^d$ be the set of $d$-tuples $\bfs$ of positive integers with $|\bfs|=n-1$. Since every element of $[d]^g$ can be written in the form of $ (1_{s_1-1},2_{s_2-1},\dots,d_{s_d-1})$, we may define a map
\begin{alignat}{4}\label{defn:rho}
    \rho:  && [d]^g \hskip1cm \ &\, \lra \DW(d,n)  \\
   && (1_{s_1-1},2_{s_2-1},\dots,d_{s_d-1})&\,\lmaps  (s_1,\dots,s_d).  \notag
\end{alignat}
It's clear that $\rho$ has an inverse so that it provides a 1-1 correspondence. Moreover,
\begin{equation*}
K^{(p)}_{d;m}(\bfb) = \calH_{m}^{(p)}(\rho (\bfb)).
\end{equation*}
Thus, by the substitution $a_j\to \gk-a_j$ we have
\begin{equation*}
V_{\gk}^{g;p} (n)=\sum_{\substack{0<a_g<\cdots<a_1<\gk\\ \bfb\in [d]^g}} K_{d;a_g p}^{(p)}(\bfb)
=\sum_{\substack{0<a_g<\cdots<a_1<\gk\\ \bfs\in \DW(d,n-1)}}\calH_{a_g p}^{(p)}(\bfs).
\end{equation*}
For each $\bfs\in\DW(d,n-1)$, let $\Gamma_d$ be its permutation group (a symmetry group of $d$ letters), $\Orb(\bfs)$ its orbit
under $\Gamma_d$, and $\Stab(\bfs)$ its stabilizer, i.e., the subgroup of all of the permutations
that fix $\bfs$. It is well-known from group theory that
$|\Orb(\bfs)|\cdot |\Stab(\bfs)|=|\Gamma_d|=d!$. Thus we have
\begin{align}
V_{\gk}^{g;p} (n)=&\, \sum_{\substack{0<a_g<\cdots<a_1<\gk\\ \bfs\in \DW(d,n-1)}}
\frac1{|\Orb(\bfs)|}\sum_{\bft\in\Orb(\bfs)} \calH_{a_g p}^{(p)}(\bft) \notag\\
=&\, \sum_{\substack{0<a_g<\cdots<a_1<\gk\\ \bfs\in \DW(d,n-1)}} \frac1{|\Orb(\bfs)|} \cdot \frac{U_{0;a_g}^{(p)}(\bfs)}{|\Stab(\bfs)|}
= \sum_{\substack{0<a_g<\cdots<a_1<\gk\\ \bfs\in \DW(d,n-1)}} \frac{U_{0;a_g}^{(p)}(\bfs)}{d!}.  \label{equ:calO}
\end{align}
Since $d=n-g-1$ and $B_{p-n}=0$ for even $n$, by Lemma \ref{lem:Uab},
\begin{equation*}
U_{0;a_g}^{(p)}(\bfs)\equiv a_g (-1)^{g+1} (d-1)! (n-1) \frac{B_{p-n}}{n} p  \pmod{p^2}.
\end{equation*}
Noticing that $|\DW(d,n-1)|=\binom{n-2}{d-1}$, we get
\begin{alignat*}{4}
V_{\gk}^{g;p}(n)\equiv &\sum_{0<a_g<\cdots<a_2<a_1<\gk} a_g
(-1)^{g+1} \binom{n-2}{d-1}\frac{n-1}{d} \frac{B_{p-n}}{n} p && \pmod{p^2}\\
\equiv &\sum_{0<a_g<\cdots<a_2<a_1<\gk} a_g
(-1)^{g+1} \binom{n-1}{g} \frac{B_{p-n}}{n} p  &&\pmod{p^2}.
\end{alignat*}
So the lemma follows from \eqref{equ:i=glemma} at once.
\end{proof}

\begin{lem}\label{lem:ga}
For any positive integers $i\le g<\gk$, we have
\begin{equation*}
\sum_{0<a_1<\cdots<a_g<\gk} a_i= i\sum_{a=1}^{\gk-1} \binom{a}{g}.
\end{equation*}
In particular, if $i=g$ then we have
\begin{equation}\label{equ:i=glemma}
  \sum_{0<a_g<\cdots<a_2<a_1<\gk} a_g = \binom{\gk}{g+1}
\end{equation}
\end{lem}
\begin{proof} Clearly
\begin{equation*}
\sum_{0<a_1<\cdots<a_i} 1=   \binom{a_i-1}{i-1}=\frac{i}{a_i} \binom{a_i}{i}
\end{equation*}
is the number of ways to choose $i-1$ distinct positive integers from $1,2,\dots, a_i-1$.
The lemma follows quickly from an induction on $g$ by using the well-known identity
\begin{equation*}
\sum_{0<a_i<a_{i+1}} \binom{a_i}{i}=\binom{a_{i+1}}{i+1}.
\end{equation*}
In particular, if $i=g$ then we may take $a_{i+1}=\gk$ to prove \eqref{equ:i=glemma}.
\end{proof}

\begin{lem}\label{lem:Mg}
We have
\begin{equation*}
M_{\gk}^{g;p}(n)\equiv 0 \pmod{p}.
\end{equation*}
\end{lem}
\begin{proof} Again we let $d=n-g-1$.
By the definition and Lemma \ref{lem:ga},
\begin{align}
M_{\gk}^{g;p}(n)= & \sum_{\substack{0<a<\gk\\ 1\le b_1\le \cdots\le b_g\le d}}
\binom{a}{g}\sum_{\substack{0< u_1<\dots<u_d<(\gk-a) p
    \\ u_1,u_2,\dots,u_d \in \calP_p}} \frac{1}{u_1\dots u_d u_{b_1}\dots u_{b_g}}
    \left(\frac{1}{u_{b_1}}+\frac{1}{u_{b_1+1}}+\right.         \notag \\
  \cdots  +\frac{1}{u_{b_2}} &\left.+\frac{2}{u_{b_2}}+\cdots+\frac{2}{u_{b_3}}+\frac{3}{u_{b_3}}+\cdots+
    \frac{g-1}{u_{b_g}}+\frac{g}{u_{b_g}}+\cdots+ \frac{g}{u_d}\right).   \label{equ:MgExtraTerms}
\end{align}
Because of the terms in the parenthesis, we see that each $\bfb\in[d]^g$
may produce more than one $p$-restricted MHSs of weight $n$. Hence,
\begin{equation*}
 M_{\gk}^{g;p}(n)= \sum_{ 0<a<\gk}\binom{a}{g} \sum_{\bfs\in \DW(d,n)} m(\bfs) \calH_{(\gk-a)p}^{(p)}(\bfs).
\end{equation*}
We now show that the multiplicity $m(\bfs)=g(g+1)/2$ for all $\bfs$. For simplicity, we set
\begin{equation*}
\bfl=(l_1,\dots,l_d)=(s_1-1,\dots,s_d-1).
\end{equation*}
The idea is to subtract 1 from a component $s_j>1$ of $\bfs$ and consider
the corresponding $\bfb(j)$ using the 1-1 correspondence $\rho$ defined by \eqref{defn:rho}.
Every such $\bfb(j)$ produced will lead to a $p$-restricted MHS $\calH_{(\gk-a)p}^{(p)}(\bfs)$
with some multiplicity due to the possible repetition of
$1/u_j$-term in the parenthesis of \eqref{equ:MgExtraTerms}.
Suppose $s_j\ge 2$. Then we get the corresponding
\begin{equation*}
\bfb(j)=(b_1,\dots,b_g)=(1_{l_1},2_{l_2},\dots,(j-1)_{l_{j-1}}, j_{l_j-1}, (j+1)_{l_{j+1}},\dots,d_{l_d}).
\end{equation*}
Set $t=l_1+\cdots+l_{j-1}.$
Then we see that $b_{t+i}=j$ for all $i=1,\dots,l_j-1$.
So the contribution to the multiplicity of $m(\bfs)$, denoted by $m_j(\bfs)$,
by this particular $\bfb(j)$ is given by the coefficient of $1/u_j$ in the above
(note that $1/u_j$ repeats $l_j$ times with increasing numerators), namely,
\begin{equation*}
m_j(\bfs)=\mu_j(\bfl):=t+\sum_{i=1}^{l_j-1} (t+i)=\Big(l_1+\dots+l_{j-1}+ \frac{l_j-1}{2}\Big) l_j.
\end{equation*}
Remarkably, this is still true even if $s_j=1$, i.e., $l_j=0$, because $\bfb(j)$ doesn't exist in this case
while $m_j(\bfs)=0$ according to the formula.

We now show that $\mu(\bfl)$ only depends on $|\bfl|=n-d=g+1$.
Indeed, let $\bfl'=(l_1-1,\dots,l_{i-1},l_i+1,l_{i+1},\dots,l_d)$ for some $i\ge 2$ and
let $r_j=\mu_j(\bfl)-\mu_j(\bfl')$. If $j=1$, we have
\begin{equation*}
r_1=\Big(\frac{l_1-1}{2}\Big) l_1- \Big(\frac{l_1-2}{2}\Big)(l_1-1)=l_1-1.
\end{equation*}
For $1<j<i$,
\begin{equation*}
r_j=\Big(l_1+l_2+\dots+l_{j-1}+ \frac{l_j-1}{2}\Big) l_j-
\Big(l_1-1+l_2+\dots+l_{j-1}+ \frac{l_j-1}{2}\Big) l_j=l_j.
\end{equation*}
For $j=i$,
\begin{multline*}
r_i=\Big(l_1+l_2+\dots+l_{i-1}+ \frac{l_i-1}{2}\Big) l_i-
\Big(l_1-1+l_2+\dots+l_{i-1}+ \frac{l_i}{2}\Big) (l_i+1)\\
=1-(l_1+\dots+l_{i-1}).
\end{multline*}
For $j>i$,
\begin{equation*}
r_j=\Big(l_1+l_2+\dots+l_{j-1}+ \frac{l_j-1}{2}\Big) l_j-
\Big(l_1-1+l_2+\dots+l_{j-1}+1+\frac{l_j-1}{2}\Big) l_j=0.
\end{equation*}
Therefore
\begin{equation*}
    \mu(\bfl)-\mu(\bfl')=\sum_{j=1}^d\Big( \mu_j(\bfl)-\mu_j(\bfl')\Big)=\sum_{j=1}^d r_j=0.
\end{equation*}
The upshot is that
$m(\bfs)=\sum_{j=1}^d m_j(\bfs)=\sum_{j=1}^d \mu_j((g+1,0,\dots,0))=\mu_1((g+1,0,\dots,0))=g(g+1)/2$
as desired. Consequently, using the idea to derive \eqref{equ:calO}, we see that
\begin{align*}
 M_{\gk}^{g;p}(n)=&\, \frac{g(g+1)}2 \sum_{ 0<a<\gk}\binom{a}{g} \sum_{\bfs\in \DW(d,n)}  \calH_{(\gk-a)p}^{(p)}(\bfs)\\
=&\, \frac{g(g+1)}{2} \sum_{ 0<a<\gk}\binom{a}{g} \sum_{\bfs\in \DW(d,n)}  \frac{U_{0;\gk-a}^{(p)}(\bfs)}{d!}\equiv 0 \pmod{p}
\end{align*}
by Lemma \ref{lem:Uab}.
\end{proof}

\begin{lem}\label{lem:Pg}
We have
\begin{equation*}
P_{\ga;\gk}^{g;p} (n)\equiv P_{0;\gk}^{g;p} (n) \pmod{p^2}.
\end{equation*}
\end{lem}
\begin{proof} As before we let $d=n-g-1$. Define
\begin{align*}
E_{\gk}^{g;p} (n):=&\,
\sum_{\substack{0<a_1<\cdots<a_g<\gk\\ 1\le b_1\le \cdots\le b_g\le d}}\
 \sum_{\substack{0< u_1<\dots<u_d<(\gk-a_g) p
    \\ u_1,u_2,\dots,u_d \in \calP_p}}
\frac{1}{u_1\dots u_d u_{b_1}\dots u_{b_g}}
    \left(\frac{1}{u_1}+\cdots+\frac{1}{u_d}\right), \\
F_{\gk}^{g;p} (n):=&\,
\sum_{\substack{0<a_1<\cdots<a_g<\gk\\ 1\le b_1\le \cdots\le b_g\le d}}\
 \sum_{\substack{0< u_1<\dots<u_d<(\gk-a_g) p
    \\ u_1,u_2,\dots,u_d \in \calP_p}}
\frac{1}{u_1\dots u_d u_{b_1}\dots u_{b_g}}
    \left(\frac{1}{u_{b_1}}+\cdots+\frac{1}{u_{b_g}}\right).
\end{align*}
Then it is easy to see that
\begin{equation}\label{equ:D=E+F}
P_{0;\gk}^{g;p} (n)-P_{\ga;g}^{\gk} (n;p) \equiv \ga p\Big(E_{\gk}^{g;p} (n)+F_{\gk}^{g;p} (n)\Big) \pmod{p^2}.
\end{equation}
Indeed, in the definition \eqref{equ:defnP} we may replace every $u_j$ by $u_j+\ga p$.
Then by geometric expansion in the $p$-adic integer ring $\Z_p$, we see that
\begin{equation}\label{equ:geomExp}
\frac1{u_j+\ga p}\equiv \frac1{u_j}\left(1-\frac{\ga p}{u_j}\right),\ \frac1{u_j+(\ga+a_i) p}\equiv\frac1{u_j}\left(1-\frac{(\ga+a_i) p}{u_j}\right) \pmod{p^2},
\end{equation}
which quickly imply \eqref{equ:D=E+F}.

We first prove that
\begin{equation}\label{equ:E=0}
E_{\gk}^{g;p} (n)\equiv 0 \pmod{p}.
\end{equation}
By the proof of Lemma \ref{lem:Vg} we see that there is a 1-1 correspondence between $[d]^g$ and $\DW(d,n-1)$,
where $[d]^g$ is the set of $g$-tuples of integers in $\{1,\dots,d\}$ and $\DW(d,n-1)\subset \N^d$
is the set of $d$-tuples $\bfs$ with $|\bfs|=n-1$. Let the height of $\bfs$, denoted by $\hht(\bfs)$,
be the number of components of $\bfs$ which are greater than 1. Let $\DW(d,n,h)$ be the subset of
height $h$ elements of $\DW(d,n)$. Since $n-d=g+1\ge 1$ the height of every element in $\DW(d,n)$ is at least 1.
Define
\begin{alignat*}{4}
    \gl_j:  && \DW(d,n-1)&\, \lra \DW(d,n)  \\
          && (s_1,\dots,s_d) \  &\,\lmaps   (s_1,\dots,s_{j-1},s_j+1,s_{j+1},\dots,s_d).  \notag
\end{alignat*}
It is obvious that the union of the images of $\gl_j$, as a multi-set, covers every element
of $\DW(d,n,h)$ exactly $h$ times. Note further that the set $\DW(d,n,h)$ is invariant under
every permutation of the components of its elements. Using the same idea to derive \eqref{equ:calO}, we get
\begin{align*}
E_{\gk}^{g;p} (n)=&\, \sum_{0<a_1<\cdots<a_g<\gk}\sum_{h=1}^d h\sum_{\bfs\in \DW(d,n,h)} \calH_{(\gk-a_g) p}^{(p)}(\bfs)\\
=&\, \sum_{0<a_1<\cdots<a_g<\gk}\sum_{h=1}^d h\sum_{\bfs\in \DW(d,n,h)}
\frac{U_{0;\gk-a}^{(p)}(\bfs)}{d!}\equiv 0 \pmod{p}
\end{align*}
by Lemma \ref{lem:Uab}.

We now prove that
\begin{equation}\label{equ:F=0}
F_{\gk}^{g;p} (n)\equiv 0 \pmod{p}.
\end{equation}
We modify the idea used in the proof of Lemma \ref{lem:Mg}. Recall that for any $\bfs=(s_1,\dots,s_d)\in \DW(d,n-1)$, we set
$\rho^{-1}(\bfs)=(1_{l_1},2_{l_2},\dots,d_{l_d})$ where $l_j=s_j-1$ for all $j=1,\dots,d$. So we argue similarly
as in the proof of Lemma \ref{lem:Mg} and see that
\begin{equation*}
F_{\gk}^{g;p}(n)= \sum_{0<a_1<\cdots<a_g<\gk} \sum_{\bfs\in \DW(d,n)} m(\bfs) \calH_{(\gk-a_g)p}^{(p)}(\bfs),
\end{equation*}
where the multiplicity
\begin{equation*}
m(\bfs)=l_1+l_2+\dots+l_d=g
\end{equation*}
which is independent of $\bfs$. Thus
\begin{align*}
F_{\gk}^{g;p}(n)=&\, \sum_{0<a_1<\cdots<a_g<\gk} g\sum_{\bfs\in \DW(d,n)}  \calH_{(\gk-a_g)p}^{(p)}(\bfs) \\
= &\, \sum_{0<a_1<\cdots<a_g<\gk} g\sum_{\bfs\in \DW(d,n)}
\frac{U_{0;\gk-a}^{(p)}(\bfs)}{d!}\equiv 0 \pmod{p}
\end{align*}
by Lemma \ref{lem:Uab}.

Finally, the lemma follows from \eqref{equ:D=E+F}, \eqref{equ:E=0} and \eqref{equ:F=0}.
\end{proof}

We are now ready to prove Proposition~\ref{prop:Pg}. By the definition, we have
\begin{align*}
P_{0;\gk}^{g;p} (n)=&\, \sum_{\substack{0<a_1<\cdots<a_g<\gk\\ 0< b_1\le \cdots\le b_g< n-g}}
\sum_{\substack{0< u_1<\dots<u_d<(\gk-a_g) p
    \\ u_1,u_2,\dots,u_d \in \calP_p}} \frac{1}{u_1u_2\dots u_{b_1}(u_{b_1}+a_1p)(u_{b_1+1}+a_1p)} \\
   &\,  \cdots \frac{1}{(u_{b_2}+a_1p) (u_{b_2}+a_2p) \dots (u_{b_g}+a_{g-1}p)(u_{b_g}+a_gp)\cdots (u_d+a_gp)}  \\
\equiv &\, V_{\gk}^{g;p}(n)-p M_{\gk}^{g;p}(n)  \pmod{p^2}
\end{align*}
by \eqref{equ:geomExp}. Thus by Lemma \ref{lem:Vg} and Lemma \ref{lem:Mg}
\begin{equation*}
P_{0;\gk}^{g;p} (n)\equiv (-1)^{g+1}\binom{\gk}{g+1}\binom{n-1}{g}\frac{B_{p-n}}{n} p \pmod{p^2}.
\end{equation*}
So Proposition~\ref{prop:Pg} follows from Lemma \ref{lem:Pg}.

We can now turn to the proof of the Main Theorem. From Theorem~\ref{thm:R1}
and Lemma~\ref{lem:Sm=Rms} we see that for all $m,n\in\N$, both $R_n^{(m,1)}$ and $S_n^{(m,1)}$ lie in the
sub-algebra $\calB$ of $\calA_1$ generated by $\calA$-Bernoulli numbers. This implies that
$S_n^{(m,2)}$ lies in $p\calB\subset \calA_2$ by Lemma~\ref{lem:Srecurrence} (ii), which in turn
yields \eqref{equ:mainS} and \eqref{equ:mainR} by Lemma~\ref{lem:Srecurrence} (iii)
and Lemma~\ref{lem:RmS1}, respectively.
We can now conclude the proof of our Main Theorem and the paper.


\begin{thebibliography}{10}


\bibitem{AIK}
T.\ Arakawa, T.\ Ibukiyama and M.\ Kaneko, Bernoulli Numbers and Zeta Functions,
with an appendix by D.\ Zagier, Monographs in Math., Springer, 2014.

\bibitem{CHQWZ}
K.\ Chen, R.\ Hong, J.\ Qu, D.\ Wang and J. Zhao,
Some families of super congruences involving alternating multiple harmonic sums, arXiv:1702.08599.

\bibitem{DilcherSl}
K.\ Dilcher and I.\ Sh. Slavutskii, 	
A Bibliography of Bernoulli Numbers,
\url{www.mscs.dal.ca/~dilcher/bernoulli.html}

\bibitem{HessamiPilehrood2Ta2013}
Kh.\ Hessami Pilehrood, T.\ Hessami Pilehrood, and R.\ Tauraso,
New properties of multiple harmonic
sums modulo $p$ and $p$-analogues of Leshchiner's series, \emph{Trans.\ Amer.\ Math.\ Soc.}, \textbf{366} (6) (2014), pp.\ 3131--3159.

\bibitem{Hoffman2015}
M.E.\ Hoffman,Quasi-symmetric functions and mod $p$ multiple harmonic sums,
\emph{Kyushu J.\ Math.} \textbf{69} (2015), pp.\ 345--366.

\bibitem{KanekoZa2013}
M.\ Kaneko and D.\ Zagier, Finite multiple zeta values, in preparation.

\bibitem{Kontsevich2009}
M.Kontsevich, Holonomic $D$-modules and positive characteristic,
\emph{Japan J.\ Math.} 4 (2009), pp.\ 1--25.

\bibitem{MTWZ}
M.\ McCoy, K.\ Thielen, L.\ Wang and J. Zhao,
A family of super congruences involving multiple harmonic sums.
\emph{Intl.\ J.\ Number Theory} \textbf{13}(1) (2017), pp.\ 109--128.

\bibitem{PetkovsekWZ1996}
M.\ Petkovsek, H.\ Wilf and D.\ Zeilberger, A=B, A K Peters/CRC Press, 1996.

\bibitem{ShenCai2012b}
Z.\ Shen and T. Cai, Congruences for alternating triple harmonic sums,
\emph{Acta Math. Sinica (Chin. Ser.)}, \textbf{55} (4) (2012), pp.\ 737--748.

\bibitem{Vermaseren1999}
J.\ A.\ M.\ Vermaseren, Harmonic sums, Mellin transforms and integrals,
\emph{Internat.\ J.\ Modern Phys.\ A} \textbf{14}(13) (1999), pp.\ 2037--2076.

\bibitem{Wang2014b}
L.\ Wang, A new curious congruence involving multiple harmonic sums,
\emph{J. Number Theory} \textbf{154} (2015), pp.\ 16--31.

\bibitem{WangCa2014}
L.\ Wang and T.\ Cai, A curious congruence modulo prime powers,
\emph{J.\ Number Theory}, \textbf{144} (2014), pp.\ 15--24.

\bibitem{Wang2015}
L.\ Wang, New congruences on multiple harmonic
sums and Bernoulli numbers. arXiv:1504.03227.

\bibitem{XiaCa2010}
B.\ Xia and T.\ Cai, Bernoulli numbers and congruences for harmonic sums,
 \emph{Int.\ J.\ Number Theory} \textbf{6} (4) (2010), pp.\ 849--855.

\bibitem{Zhao2007b}
J.\ Zhao, Bernoulli numbers, Wolstenholme's Theorem, and $p^5$ variations of Lucas' Theorem,
\emph{J.\ Number Theory} \textbf{123} (2007), pp.\ 18--26.

\bibitem{Zhao2008a}
J.\ Zhao, Wolstenholme type theorem for multiple harmonic
sums, \emph{Int.\ J.\ Number Theory} \textbf{4} (1)(2008), pp.\ 73--106.

\bibitem{Zhao2011c}
J.\ Zhao, Mod $p$ structure of alternating and non-alternating multiple
harmonic sums.
\emph{J.\ Th\'eor.\ Nombres Bordeaux} \textbf{23} (1) (2011), pp.\ 259--268. (\textbf{MR} 2780631)

\bibitem{Zhao2014}
J.\ Zhao, Congruences involving multiple harmonic sums and finite multiple zeta values. arXiv:1404.3549.

\bibitem{Zhao2015a}
J.\ Zhao,
Multiple Zeta Functions, Multiple Polylogarithms and Their Special Values, Series on
Number Theory and Its Applications, vol.\ \textbf{12},
World Scientific Publishing Co. Pte. Ltd., Hackensack, NJ, 2016. %ISBN 9814689394

\bibitem{ZhouCa2007}
X.\ Zhou and T.\ Cai, A generalization of a curious congruence on harmonic sums,
\emph{Proc.\ Amer.\ Math.\ Soc.} \textbf{135} (2007), pp.\ 1329--1333.

\end{thebibliography}
\end{document}